\documentclass[12pt, reqno]{amsart}
\usepackage{amsmath, amsthm, amscd, amsfonts, amssymb, graphicx, color}
\usepackage[bookmarksnumbered, colorlinks, plainpages]{hyperref}
\hypersetup{colorlinks=true,linkcolor=red, anchorcolor=green, citecolor=cyan, urlcolor=red, filecolor=magenta, pdftoolbar=true}

\newtheorem{theorem}{Theorem}[section]

\newtheorem{proposition}[theorem]{Proposition}

\theoremstyle{definition}
\newtheorem{definition}[theorem]{Definition}
\newtheorem{example}[theorem]{Example}

\theoremstyle{remark}

\numberwithin{equation}{section}

\begin{document}

\setcounter{page}{1}

\title[A generalization of  $b$-weakly compact operators]{A generalization of  $b$-weakly compact operators}

\author[Kazem Haghnejad Azar]{Kazem Haghnejad Azar}

\address{Department of Mathematics, University of Mohaghegh Ardabili, Ardabil, Iran.}
\email{\textcolor[rgb]{0.00,0.00,0.84}{haghnejad@uma.ac.ir;
}}


\let\thefootnote\relax\footnote{Copyright 2018 by the Tusi Mathematical Research Group.}



\date{Received: xxxxxx; Revised: yyyyyy; Accepted: zzzzzz.
\newline \indent $^{*}$Corresponding author}

\begin{abstract}
A. Bahramnezhad and K. Haghnejad Azar    introduced the classes of  $KB$-operators and $WKB$-operators, and they studied some of theirs properties. In the present paper, we give  answer for an open problem from that paper, which two classifications of operators, $b$-weakly compact operators  and  $KB$-operators   are different.  A continuous operator $T$ from a normed vectoe lattice $E$ into a normed space $X$ is said to be $KB$-operator (respectively, $WKB$-operator)  if $\{Tx_n\}_n$ has a norm (respectively, weak) convergent subsequence in $X$ for every positive increasing sequence $\{x_n\}_n$ in the closed unit ball $B_E$ of $E$. We investigate some other properties of $KB$-operators and its relationships with $b-$weakly compact operators.
\end{abstract} \maketitle

\section{\textbf{Introduction and preliminaries}}
A subset $A$ of a vector lattice $E$ is called $b$-order bounded in $E$ if it is order bounded in $E^{\sim\sim}$.  An operator $T:E\rightarrow X$,  mapping each $b$-order bounded subset of $E$ into a relatively weakly compact subset of $X$ is called a $b$-weakly compact operator. Alpay, Altin and Tonyali introduced the class of $b$-weakly compact operators for vector lattices having separating order duals \cite{2}. In  \cite{3}, Alpay and Altin proved that a continuous operator $T$ from a Banach lattice $E$ into a Banach space $X$ is $b$-weakly compact if and only if $\{Tx_n\}_n$ is norm convergent for each $b$-order bounded increasing sequence $\{x_n\}_n$  in  $E^+$ if and only if  $\{Tx_n\}_n$ is norm convergent to zero for each $b$-order bounded disjoint sequence $\{x_n\}_n$ in $E^+$.  Authors in \cite{6} proved that an operator $T$ from a Banach lattice $E$ into a Banach space $X$ is $b$-weakly compact if and only if $\{Tx_n\}_n$  is norm convergent for every positive increasing sequence $\{x_n\}_n$ of the closed unit ball  $B_E$ of $E$. The aim of this paper is studied  the  classes of operators on Banach lattices are called $KB$-operators and $WKB$-operators. We will investigate on theirs properties. A continuous operator $T$ from a Banach lattice $E$ into a normed space $X$ is said to be $KB$-operator ( $WKB$-operato),  if $\{Tx_n\}_n$ has a norm  ( weak) convergent subsequence in $X$ for every positive increasing sequence $\{x_n\}_n$ in the closed unit ball $B_E$ of $E$, see \cite{9}.  To state our results, we need to fix some notation and recall some definitions.\\ 
  A Banach lattice $E$ is is said to be an $AM$-space if for each $x,y\in E$ such that $|x|\wedge |y|=0$, we have $\|x+y\|= max \{\|x\|, \|y\|\}$. A Banach lattice $E$ is an $AL$-space if its topological dual $E^\prime$ is an $AM$-space. A Banach lattice $E$ is said to be $KB$-space whenever each increasing norm bounded sequence of $E^+$ is norm convergent. An operator $T: E\rightarrow F$ between two Riesz spaces  is positive if $T(x)\geq 0$ in $F$ whenever $x\geq 0$ in $E$. Note that each positive linear mapping on a Banach lattice is continuous.  An operator $T$ from a Banach space $X$ into a Banach space $Y$ is compact (resp. weakly compact) if $\overline{{T(B _ X)}}$  is compact (resp. weakly compact) where $B _ X$ is the closed unit ball of $X$.   
For terminology concerning Banach lattice theory and positive operators, we refer the reader to the excellent book of \cite{1}.\\
  As $b$-weakly compact operators \cite{7}, the class of $KB$-operators and $WKB$-operators does not satisfy duality property. In fact the identity operator of the Banach lattice $\ell^1$ is a $KB$-operator (respectively,  $WKB$-operator); but its adjoint which is the identity operator of the Banach lattice $\ell^ \infty $, is not a  $KB$-operator (respectively,  $WKB$-operator). Conversely, the identity operator of the Banach lattice $c_0$ is not a $KB$-operator (respectively,  not $WKB$-operator); but its adjoint, which is the identity operator of the Banach lattice $\ell^1$, is a $KB$-operator (respectively,  $WKB$-operator). Recall that a Banach space is said to have Schur property whenever every weak convergent sequence is norm convergent, i.e., whenever $x_n \xrightarrow{w}0$ implies $\|x_n\| \rightarrow 0$. Let $E$, $F$ be Banach lattices. If either $E$ or $F$ has the Schur property then $L(E,F)=W_b (E,F)$ \cite{3}. It is also clear that  for a Banach lattice  $E$ and  a Banach space $X$ with Schur property,  every  $WKB$-operator $T: E\rightarrow X$ is a $KB$-operator. Let $E$ be a vector lattice. A sequence  
 $\{x_n\}_1^ \infty\subset E$
 is 
 called  order  convergent  to 
 $x$
 as 
 $n \to  \infty$
 if there exists a sequence  
 $\{y_n\}^\infty_1$
 such  that  
 $y_n \downarrow 0$
 as
 $n  \rightarrow \infty$
 and 
 $|x_n -x| \leq  y_n$
 for  all 
 $n$. We will write $x_n\xrightarrow{o_1}x$ when $\{x_n\} $ is order convergent to $x$. A sequence $\{x_n\}$ in a vector lattice  $E$ is    strongly order convergent  to  $x\in E$, denoted by $x _n \xrightarrow{o_2}x$ whenever there exists a net $\{y_\beta\}_{\beta\in \mathcal{B}}$  in $E$ such that $y_\beta \downarrow 0$ and that for every $\beta\in \mathcal{B}$, there exists $n_0$ such that  $|x_n -x| \leq y_\beta$ for all $n\geq n_0$. It is clear that every order convergent sequence is strongly order convergent, but two convergence are different, for information see, \cite{1b}. A net $ (x_{\alpha})_{\alpha}$ in Banach lattice $E$ is unbounded norm convergent (or, $un$-convergent for short) to $ x \in E $ if $ | x_{\alpha} - x | \wedge u\xrightarrow{\|.\|} 0$ for all $ u \in E^{+} $.  We denote this convergence by $ x_{\alpha} \xrightarrow{un}x $. This convergence has been introduced and studied in \cite{9d, 10a}.

\section{Main Results}
In each parts of this manuscript,  $E$ is a normed vector lattice and $X$ normed space. The definition of $b-$weakly compact operator holds whenever $E$ is a  normed vector lattice and $X$ normed space. The collections of $KB$-operators,  $WKB$-operators,  b-weakly compact operators,   order weakly compact operators,  weakly compact operators and compact operators will be denoted by
 $L_{KB}(E,X)$, $W_{KB}(E,X)$, $W_b(E,X)$, $W_o(E,X)$, $W(E,X)$ and $K(E,X)$, respectively,  whenever there is not confused. By notice an example from  \cite{10}, page 95, the classification of order compact or order weakly compact operators from Banach lattice $E$ into Banach space $X$, in general, are not subspace of  $L_{KB}(E,X)$, but by Theorem 3.4.4, from \cite{10}, every interval-bounded and order weakly compact operator $T:E\rightarrow X$ is $b-$weakly compact operator, and so  $KB-$operator. On the other hand, if $e\in E$ is an atom, then $L_{KB}(E_e,F)=W_o(E_e,F)$. We have the following relationships between these spaces:
$$K(E,X) \subseteq W(E,X)\subseteq W_b(E,X)\subseteq L_{KB}(E,X)\subseteq W_{KB}(E,X).$$

Let $E$  be normed vector lattices, $X$ normed space and $T: E\rightarrow X$ be a positive operator. By using Prpostion 2.1 from \cite{9},  $ T\in W_b (E,X)$ iff $ T\in L_{KB}(E,X)$ iff $T\in W_{KB}(E,X)$ as follows. 
\begin{proposition}\cite{9} \label{2.1}
Let $E$ and $F$ be Banach lattices and $T: E\rightarrow F$ be a positive operator. Then the following statements are equivalent.
\begin{enumerate}
\item $T$ is $b$-weakly compact.
\item  $T$ is $KB$-operator.
\item  $T$ is $WKB$-operator.
\end{enumerate}
\end{proposition}

 In the following example, we give  answer for an open problem from \cite{9}, which  two classifications of operators $W_b(E,X)$  and  $L_{KB}(E,X)$ are different. In the other words, there is a $KB$-operator $T$ from normed vector lattice $E$ into Banach space $X$ that is not $b-$weakly compact operator. 
  
\begin{example}\label{2.2}
 In the following, we show that the inclusion $W_b(E,F)\subseteq L_{KB}(E,F)$ may be proper whenever $E$ is normed vector lattice and $F$ is a Banach space.
 \begin{proof}
 Authors in Theorem 2.10 of \cite{6} proved that an operator $T$ from Banach lattice $E$ into Banach lattice $F$ is $b-$weakly compact if and only if $(T(x_{n}))_{n=1}^{+\infty}$ is norm convergent for every positive increasing sequence $(x_{n})$ in the closed unit ball $B_{E}$ of $E$. Now this theorem holds when  $E$ is normed vectoe lattice.\\
Suppose that  $E=\langle \lbrace \chi_{[0,r]}\ :\ r \in \mathbb{R}^{+}\rbrace \rangle$ where $\chi_{[0,r]}$ is a characteristic function on $[0,r]$.  It is clear that $E$ is a subspace of  $L^{\infty}([0,+\infty))$ and 
 with $\Vert . \Vert_{\infty}$ is normed  vector lattice. 
We define an operator  $T:E \rightarrow l^{\infty}$ with 
\begin{align*}
T(f)=( \int_{0}^{+\infty} f(x) \cos n \pi x dx)_{n=1}^{+\infty} \quad\quad\text{for~ all} \quad f \in E. 
\end{align*}
$T$ is a bounded linear operator of course 
\begin{equation*}
|\int_{0}^{+\infty}f(x)\cos n\pi x dx| \leq \int_{0}^{+\infty}| f(x) | dx <+\infty.
\end{equation*}
Now the proof will be based upon three claims as follows:\\
\textbf{claim 1:} $T$ is not positive operator.\\
 If we set $f=\chi_{[0,\frac{1}{2}]}$, then $Tf= (\frac{1}{n\pi} \sin \frac{n\pi}{2})_{n=1}^{+\infty}$ which shows that $T$ is not positive operator.\\
 \textbf{claim 2:} $T$ is not $b-$weakly compact operator.\\
  Let $f_{m}=\chi_{[0,m+\frac{1}{2}]}$ where $m \in \mathbb{N}$. Then $0\leq f_{m}\uparrow$ and $\mathop{\text{sup}}_{m} \Vert f_{m} \Vert_{\infty}<+\infty$, but $(Tf_{m})_{m=1}^{+\infty}$ is not convergent.\\
 \textbf{claim 3:} $T$ is a $KB-$operator.\\
  Let $\lbrace f_{m}\rbrace \subseteq E$, $0\leq f_{m}\uparrow$ and  $\mathop{\text{sup}}_{m} \Vert f_{m} \Vert_{\infty}<+\infty$. We write $Tf_{m}=a_{m,n}$ where $a_{m,n}=\int_{0}^{+\infty}f_{m}(x) \cos nx dx$.\\
 For fix $n \in \mathbb{N}$, the sequence $(a_{m,n})_{m=1}^{+\infty}$ is norm bounded in $\mathbb{R}$, and so it has a subsequence  $(a_{k_{m},n})_{m=1}^{+\infty}$ which is convergent. It is clear $0 \leq f_{k_{m}} \uparrow$ and $\mathop{\text{sup}}_{m} \Vert f_{k_{m}} \Vert_{\infty}<+\infty$. It follows that $(a_{k_{m},n+1})_{m=1}^{+\infty}$ is norm bounded in $\mathbb{R}$, and so it has also subsequence  $(a_{(k+1)_{m},n+1})_{m=1}^{+\infty}$ which is convergent. Thus $(a_{m_{m},n})_{m=1}^{+\infty}$ is a subsequence of $(a_{m,n})_{m=1}^{+\infty}$ which is convergent for each $n \in \mathbb{N}$. On the other hand, $(Tf_{m_{m}})_{m=1}^{+\infty}$ is subsequence of $(Tf_{m})_{m=1}^{+\infty}$ which is convergent. Now, $T$ is a $KB-$operator.
 \end{proof}
 \end{example}
Proposition \ref{2.1}, shows that   $L_{KB}(E,F)=W_b(E,F)$ whenever  $L_{KB}(E,F)$ is vector lattice. Consequently,  the above example implies that in general  $L_{KB}(E,F)$ is not vector lattice. On the other hand, in general, $L_{KB}(E,F)$ is not order dense in $B(E,F)$. Note that if we set $E=F=c_0$ and $T_n\in L_{KB}(c_0,c_0)$ defined  by $T_n(x_1,x_2,...)=(x_1,x_2,...,x_n,0,0,...)$, then 
$0\leq T_n\uparrow I_{c_0}$. But  $I_{c_0}$ is not $KB$-operator. Then  by  Proposition \ref{2.1},  $W_b(E,X)$, in general,  is also not order dense in $B(E,F)$.\\
In the following we show that $L_{KB}(E,F)$ is a norm closed subspace of $B(E,F)$ whenever $F$ is a Banach lattice.

 \begin{theorem}
The collection of all $KB$-operators from $E$ into Banach lattice $F$ is norm closed vector subspace of $B(E,F)$. 
 \end{theorem}
 \begin{proof}
 Let $\{T_n\}\subseteq B(E,F)$  be a sequence of $KB$-operators such that
 \begin{equation*}
  \Vert T_n-T\Vert\rightarrow 0 
 \end{equation*}  
  holds in $B(E,F)$ and let $\{x_m\}$ be a positive increasing sequence in $E$ with $\text{sup}_m\Vert x_m \Vert<\infty$. Since $T_1\in L_{KB}(E,F)$, there is a subsequence $\{x_{m,1}\}$ of  $\{x_m\}$ which $\{T_1x_{m,1}\}$ is norm convergent in $F$. Choose subsequence  $\{x_{m,n}\}$ and 
  $\{x_{m,{n+1}}\}$ of $\{x_m\}$ as follows
 \begin{enumerate}
\item $\{x_{m,{n+1}}\}$ is a subsequence of $\{x_{m,n}\}$ for each $n\in\mathbb{N}$..
\item $\{Tx_{m,n}\}_m$ is norm convergent in $F$ for each $n\in\mathbb{N}$.
 \end{enumerate}
 It follows that $\{T_nx_{m,m}\}_m$ is norm convergent   for each $n\in\mathbb{N}$.
From following inequalities  
 \begin{equation*}
 \Vert Tx_{m,m}-Tx_{k,k} \Vert\leq \Vert Tx_{m,m}-T_nx_{m,m}\Vert+\Vert T_nx_{m,m}-T_kx_{m,m}\Vert\\  
 +\Vert T_kx_{m,m}-Tx_{m,m}\Vert,
\end{equation*}

 the sequence $\{Tx_{m,m}\}_m$ is a norm Cauchy, and so   norm  convergent in $F$.
 \end {proof}
 
\begin{theorem}
Let $E$ and $F$ be Banach lattices where $E$ has order continuous norm. Let $G$ be a sublattice order dense in $E$ and $T$ be a positive operator from $E$ into $F$.  If $T\in L_{KB}(G,F)$, then $T\in L_{KB}(E,F)$.
\end{theorem} 
\begin{proof}
By using Propostion \ref{2.1} from \cite{9}, it is enough to give a proof for $b-$weakly compact operators. Let $\{x_n\}$ be a positive increasing sequence in $E$ with $\text{sup}_n\Vert x_n \Vert<\infty$. Since $G$ is order dense in $E$, by Theorem 1.34 from \cite{1}, we have
\begin{equation*}
\{y\in G:~0\leq y\leq x_n\}\uparrow x_n,
\end{equation*}
for each $n$. Let $\{y_{mn}\}_{m=1}^\infty\subset G$ with $0\leq y_{mn}\uparrow x_n$ for each $n$. Put $z_{mn}=\vee_{i=1}^n y_{mi}$ and  $0\leq T\in L(E,F)$. Its follows that $z_{mn}\uparrow_m x_n$ and $\text{sup}_{m,n}\Vert z_{m,n}\Vert\leq \text{sup}_n\| x_n\|<\infty$. Now, if $T\in W_b(G,F)$, then $\{Tz_{mn}\}$ is norm convergent to some point $y\in F$. Now, we have the following inequalities
\begin{equation*}
\|Tx_n-Tz_{mn}\|\leq\|T\|\|x_n-z_{mn}\|\leq \|T\|\|x_n-y_{mn}\|\rightarrow 0.
\end{equation*}
Thus by the following inequality proof holds
 \begin{equation*}
\|Tx_n-y\|\leq\|Tx_n-Tz_{mn}\|+\|Tz_{mn}-y\|.
\end{equation*}
\end{proof} 
 
By notice to Remark 2.26 from \cite{9}, $L_{KB}(E,F)$ is not vector lattice whenever $E$ and $F$ are Banach lattices. On the other hand, if $T\in L_{KB}(E,F)$, in general, the modulas of $T$ need not be a $KB-$operators. Now in the following theorem, we show that $T\in L_{KB}(E,F)$ whenever $\vert T\vert\in L_{KB}(E,F)$.

\begin{theorem}\label{t:2.6}
Let $E$ and $F$ be normed vector lattices. We have the following assertions.
\begin{enumerate}
\item If  $T:E\rightarrow F$ is an order bounded operator and $F$ is $KB-$space, then $T$ and $\vert T\vert$ are $KB-$operators.
\item If $\vert T \vert$ is $KB-$operator or $b-$weakly compact operator, then
 \begin{equation*}
 T \in L_{KB}(E,F) \cap W_{b}(E,F)
 \end{equation*}
 \end{enumerate}
 \end{theorem}
 \begin{proof}
 \begin{enumerate}
\item Let $x_{n}\uparrow$ and $ \mathop{\text{sup}}_{n}\Vert x_{n}\Vert<\infty$. Since $T^{+}$ is positive, by Theorem 4.3 \cite{1}, $T^{+}$ is norm bounded. It follows that $\lbrace T^{+}x_{n} \rbrace$ is norm bounded.
 Thus $T^{+}$ is $KB-$operator. In similar way, $T^-$ is also $KB$-operator.   
 It follows that  $T$ and $\vert T\vert$  are $KB-$operators.
 \item Since $0 \leq T^- , T^{+}\leq \vert T \vert$, then $T^-$ and $T^+$ are $KB-$operators. By Proposition 2.9, from \cite{9},  $T$ is $KB-$operators. Similar argument holds for $b-$weakly compact operators, and so by using proposition \ref{2.1},  proof holds.
\end{enumerate}  
 \end{proof} 
 
Theorem \ref{t:2.6} shows that the modulus of the operator $T$ in example of page 3 is not $KB-$operator or $b-$weakly compact operator.\\
  Note that each weakly compact operator is a $KB$-operator but the converse may be false in general. For example, the identity operator $I:L^1 [0,1]\rightarrow L^1 [0,1]$ is a $KB$-operator but is not weakly compact.\\
Let  $E$ and $F$ be two Banach lattices such that the norm of $E^ \prime $ is order continuous. Then,  by using Proposition \ref{2.1} and  \cite [Theorem 2.3]{8},  it is clear that each positive  $KB$-operator $T:E\rightarrow F$ is weakly compact. Now in the following, we show that $B(E,F)=L_{KB}(E,F)$ whenever $E$ has order unit and  order continuous norm.

 \begin{theorem}\label{thm}
 Let $E$ be a normed vectoe lattice and $F$ Dedekind Banach Lattice. By one of the following conditions
 \begin{equation*}
 B(E,F)=L_{KB}(E,F)
 \end{equation*}
 \begin{enumerate} 
\item  $E$ has order unit and  order continuous norm.  
 
\item $E$ and $E^{\prime}$ have order continuous norm and $F$ has Schur property
\end{enumerate} 
  \end{theorem}
 \begin{proof}
 \begin{enumerate} 
\item
Let $T\in B(E,F)$ and $x_{n}\uparrow$  and  $\mathop{\text{sup}}_{n} \|x_{n}\|<\infty  $. Let $e \in E^{+}$ be an order unit for $E$. For each $x \in E$, the norm 
 \begin{equation*}
   \|x\|_{\infty}=\inf \lbrace\lambda >0 : \vert x \vert \leq \lambda e\rbrace
 \end{equation*}
 is equivalent to the original norm $\Vert .\Vert$ of $E$. It follows that 
\begin{equation*}
 \mathop{\text{sup}}_{n} \Vert x_{n}\Vert_{\infty}=\mathop{\text{sup}}_{n} \|x_{n}\Vert.
 \end{equation*}
  For each $n\in \mathbb{N} $, there exists $\lambda_n>0$ such that $\lambda_{n} \leq \Vert x_{n}\Vert_{\infty}+1$ where $\vert x_{n}\vert \leq \lambda_{n}e$. It follows that 
\begin{equation*}
0<\lambda=\text{sup} \lambda_{n} \leq \mathop{\text{sup}}_{n} \Vert x_{n}\Vert_{\infty}+1 < \mathop{\text{sup}}_{n} \Vert x_{n}\Vert +1 <\infty.
\end{equation*}
Then $\lbrace x_{n}\rbrace \subseteq [- \lambda e , \lambda e]$. Since $E$ has order continuous norm, by using theorem 4.9 from \cite{1}, it follows that $ [- \lambda e , \lambda e]$ is weakly compact. On the other hand, since $E$ has order unit, $T$ is order bounded, and so $T^+$ exists. Thus $T^{+}( [- \lambda e , \lambda e])$ is weakly compact, since $T^{+}$ is weak to weak continuous. As $\lbrace T^{+}x_{n}\rbrace \subseteq  T^{+}( [- \lambda e , \lambda e]) $, there is subsequence $ \lbrace T^{+}x_{n_{j}} \rbrace$ from $ \lbrace T^{+}x_{n} \rbrace$ which is weak convergent to some point $z \in F$. Since $(T^{+}x_{n})_{n}$ is an increasing sequence, then by Theorem 1.4.1 from \cite{10}, $T^{+}x_{n}$ is norm convergent to $z$ in $F$. Thus $T^{+}\in L_{KB}(E,F)$.
By using equality $T=T^{+}-T^{-}$, we conclude that $T \in L_{KB}(E,F)$.
\item
Let $T \in B(E,F)$ and $(x_{n}) \subseteq E$ be increasing positive sequence where $$ \mathop {\text{sup}}_{n} \Vert x_{n} \Vert <+\infty.$$ By Theorem 4.25 from \cite{1}, there is a weakly Cauchy subsequence $\lbrace x_{n_{j}} \rbrace$  in $E$. It follows that $\lbrace Tx_{n_{j}} \rbrace$ is weakly Cauchy in $F$. Since $F$ has Schur property, $\lbrace Tx_{n_{j}} \rbrace$ is norm Cauchy in $F$, and so   convergence in $F$. Thus $T \in L_{KB}(E,F)$.
\end{enumerate} 
\end{proof}

 \begin{theorem}\label{2.6}
 Assume that an order bounded operator $T:E \rightarrow F$ between two Banach lattices has preserves disjointness. Then $T \in W_{b}(E,F)$ if and only if $T \in L_{KB}(E,F)$.
  \end{theorem}
 \begin{proof}
 By using Theorem 2.40 from \cite{1}, $T$ is $KB-$operator if and only if $\vert T \vert$ is $KB-$operator, 
 and by proposition {2.1} from  \cite{9} it is equivalent to that $\vert T \vert$ is $b-$weakly compact operator.
  Another using Theorem 2.40 \cite{1} shows that $T$ is $b-$weakly compact and proof follows.
 \end{proof}

Let $E$ and $F$ be Banach lattices and let $T\in L_{KB}(E,F)$. By Proposition 3.4 from \cite{9d},  $\{Tx_n\}_n$ has an order convergence subsequence in $F$ for every positive increasing sequence $\{x_n\}_n$ in the closed unit ball $B_E$ of $E$. Now, in the following we introduce two  classifications of operators which are generalization of order and strongly order continuous operators and we establish the relationships between them and positive $KB-$operators.

\begin{definition}
Let $E$ and $F$ be vector lattice.  $L^{(1)}_c(E,F)$ (resp. $L^{(2)}_c(E,F)$) is the collection of operators $T\in L_b(E,F)$, which $x_n\xrightarrow{o_1}0$ ($x_n\xrightarrow{o_2}0$) \text{implies} $Tx_{n_k} \xrightarrow{o_1}0$ (resp. $Tx_{n_k} \xrightarrow{o_2}0$)
\text{whenever} $\{x_{n_k}\}$ is a subsequence of $\{x_{n}\}$.
\end{definition}
In \cite{1b}, there are some examples which shows that two classifications of operators $L^{(1)}_c(E,F)$ and $L^{(2)}_c(E,F)$ are different.

\begin{theorem}\label{2.8}
Let $E$, $F$ be a Banach lattices and $E$ has order continuous norm. Then
\begin{enumerate}
\item  $W_b(E,F)^+= L_{KB}(E,F)^+ \subseteq L^{(2)}_c(E,F)$.
\item If $L_{KB}(E,F)$  is vector lattice and  $F$  Dedekind complete, then $L_{KB}(E,F)$ is an ideal in $ L^{(1)}_c(E,F)= L^{(2)}_c(E,F)$.
\end{enumerate}
\end{theorem}
\begin{proof}
\begin{enumerate}
\item Let $T$ be a positive $KB-$operator and $\{x_n\}\subset E$  strongly order convergence in $E$. Without lose generality, we set $0\leq x_n\xrightarrow{o_2} 0$, which follows $\{x_n\}$ is norm convergent to zero. Set $\{x_{n_j}\}$ as subsequence with $\sum_{k=1}^{+\infty}\Vert x_{n_j} \Vert< +\infty$. Define $y_m=\sum_{j=1}^m  x_{n_j}$. Then $0\leq y_m\uparrow$ and $\text{sup}_m\Vert y_m\Vert<\infty$. Since $T$ is $KB-$operator, $\{Ty_m\}$ is norm convergent to some point $z\in F$. Now by \cite{9g}, page 7, it has a subsequence as $\{Ty_{m_k}\}$ which  is strongly order convergent to $z\in F$. Thus there is $\{z_\beta\}\subset F^+$ and that for each $\beta$ there exists $n_0$ $\vert Ty_{m_k}-z\vert\leq z_\beta\downarrow 0$  whenever  $k\geq n_0$. If we set $n_0\leq k'\leq k$, then  we  the following inequalities
\begin{align*}
0&\leq Tx_{n_{m_k}}\leq \vert Ty_{m_k}-Ty_{m_{k'}}\vert\leq
\vert Ty_{m_k}-z\vert +\vert Ty_{m_{k'}}-z\vert\\
&\leq z_\beta+z_\beta\downarrow 0, 
\end{align*}
 shows that $T\in L^{(2)}_c(E,F)$ and proof immediately follows. 
 \item Since $L_{KB}(E,F)$ is a vector lattice,   by equality $T=T^+-T^-$ and Theorem 1.7 from \cite{1b}, we have $ L^{(2)}_c(E,F)=L^{(1)}_c(E,F)$ and by part (1), $L_{KB}(E,F)$ is a subspace of $ L^{(1)}_c(E,F)$. By Proposition 2.9, from \cite{9}, it is also obvious that   $L_{KB}(E,F)$ is an ideal in $ L^{(1)}_c(E,F)$.

\end{enumerate}
\end{proof}

\noindent {\bf Question.} By conditions Theorem \ref{2.8}, Is  $L_{KB}(E,F)$ a band in $ L^{(1)}_c(E,F)=L^{(2)}_c(E,F)$?

\begin{definition}
\begin{enumerate}  
\item An operator $T:E\rightarrow F$ between two normed vector lattice is unbounded $b$-weakly compact if $\{Tx_n\}$ is $un-$convergent for every positive increasing sequence $\{x_n\}_n$ in the closed unit ball $B_E$ of $E$.
\item An operator $T:E\rightarrow F$ between two normed vectoe lattice is unbounded $KB$-operator if  $\{Tx_n\}$ is $un-$convergent for every positive increasing sequence $\{x_n\}_n$ in the closed unit ball $B_E$ of $E$.
\end{enumerate}
\end{definition}
For normed vector lattices $E$ and $F$, the collection of unbounded\\ $KB-$operators (resp. $b-$weakly compact operators) will be denoted by\\ $L_{UKB}(E,F)$ (resp. $UW_b(E,F)$).\\
As example of page 3,  two classifications of the above operators are different.
 If a Banach lattice $F$ has strong unit, by using Theorem 2.3  \cite{10a}, we have  $L_{KB}(E,F)=L_{UKB}(E,F)$ and  $W_b(E,F)=UW_b(E,F)$.  

It is clear that every $KB-$operators (or $b-$weakly compact operators) are unbounded  $KB-$operators (or $b-$weakly compact operators), but the following example shows that  the converse in general, not holds.
   
\begin{example}\label{Ex:12}
Let $I_{c_0}$ be an identity mapping from $c_0$ into itself.  Then $I_{c_0}$ is an unbounded $KB$-operator and unbounded $b$-weakly compact operator.  But $I_{c_0}$ is not $KB$-operator or $b$-weakly compact operator.
\end{example}
Let $X$ be a Banach space, $F$ a Banach lattice, and $T\in L(X,F)$. We say that $T$ is (sequentially) $un-$compact, if for every bounded net $\{x_\alpha\}$ (resp. $\{x_n\}$ its image has a subnet (resp. subsequence) which is $un-$convergent. The collections of $un-$compact (resp. sequentially $un-$compact) will be denoted by $KU(X,F)$ (resp. $K^\sigma U(X,F)$), that is 
\begin{align*}
KU(X,F)&=\{T:X\rightarrow F\mid~T~\text{is} ~un-\text{convergent}\},\\
K^\sigma U(X,F)&=\{T:X\rightarrow F\mid~T~\text{is}~\text{sequentially} ~un-\text{convergent}\}.
\end{align*}
It is clear that $ K^\sigma U(E,F)\subseteq L_{UKB}(E,F)$ where $E$ and $F$ are Banach lattices.  As  example of page 3, this inclusion may be proper.

\begin{theorem}
Let $I$ be an ideal in Banach lattice $E$ and $T\in B(E,F)$, where $F$ is Banach lattice. If $T\in L_{KB}(I,F)$ is surjective homomorphism, then $T\in L_{UKB}(E,F)$. 
\end{theorem}
\begin{proof}
Let $\{x_n\}$ be a positive increasing sequence in $E$ with $\text{sup}_n\Vert x_n \Vert<\infty$ and let $x\in I$. Then $ x_n\wedge x\in I$, $ 0\leq x_n\wedge x\uparrow$ and $\text{sup}_n\Vert x_n\wedge x \Vert\leq\Vert x_n\Vert<\infty$. Since $T\in L_{KB}(I,F)$, there is a subsequence $\{x_{n_j}\}$ which $\{T(x_{n_j}\wedge x)\}$  is convergent for each $x\in I$. As $T$ is homomorphism and surjective,  $\{T(x_{n_j})\wedge y\}$ is convergent for all $y\in F$ and proof follows.
\end{proof}

\begin{theorem}
Let $E$ be a $KB-$space and $T\in B(E,F)$ be an surjective homomorphism, where $F$ is Banach lattice. Then $T\in L_{UKB}(E,F)$
\end{theorem}
\begin{proof}
Let $\{x_n\}\subset E^+$ be increasing sequence. Since  for each $x\in E^+$ $\text{sup}\Vert x_n\wedge x\Vert<\infty$, it follows $\{x_n\wedge x\}$ is norm  convergent, and so $\{T(x_n\wedge x)\}$ is norm convergent for each $x\in E^+$. As $T$ is surjective homomorphism, proof follows.
\end{proof}

\begin{theorem}
Let $E$ has order continuous norm and atomic and $F$ normed lattice. Then 
$$W_{KB}(E,F)\subseteq L_{UKB}(E,F).$$
\end{theorem}
\begin{proof}
Let $\{x_n\}\subset E^+$ be increasing sequence and $\text{sup}\Vert x_n\Vert<\infty$. If $T\in W_{KB}(E,F)$, there is a subsequence $\{x_{n_j}\}$ which  $\{Tx_{n_j}\}$ is weak convergent to some  point $y\in F$. By Proposition 6.2, from \cite{9d}, $\{Tx_{n_j}\}$ is unbounded norm convergent to $y$. Thus
$T\in L_{UKB}(E,F)$. 
\end{proof}

\begin{theorem}
Assume that $E$ and $F$ are Banach lattices and $F$ has order continuous norm. Then 
\begin{enumerate}
\item $W_{KB}(E,F)^+\subseteq L_{UKB}(E,F)$.
\item If $T:E\rightarrow F$ is surjective lattice homomorphism, then $T\in L_{UKB}(E,F)$.
\end{enumerate}
\end{theorem}
\begin{proof}
Let $\{x_n\}\subset E^+$ be increasing sequence and $\text{sup}\Vert x_n\Vert<\infty$. Then
\begin{enumerate}
\item if $T\in W_{KB}(E,F)^+$,  $\{Tx_n\}$ is weakly convergence in $F$. By Proposition 6.3 from \cite{9d}, $\{Tx_n\}$ is unbounded norm convergence in $F$, and so $T\in L_{UKB}(E,F)$. 
\item let $x\in E^+$. Set $y_n=x_n\wedge x$, which follows that $y_n\uparrow\leq x$ and $\text{sup}_n\Vert y_n\Vert\leq \Vert  x \Vert$. Since $T$ is lattice homomorphism, $T$ is positive, which follows  $Ty_n\uparrow\leq Tx$. By using Theorem 4.11 \cite{1}, $\{Ty_n\}$ norm Cauchy, and so norm convergence in $F$. On the other hand, $T(x_n\wedge x)=Tx_n\wedge Tx$ is norm convergent and proof follows.
\end{enumerate}
\end{proof}
In the preceeding theorem, if $W_{KB}(E,F)$  is a vector lattice, we conclude that $W_{KB}(E,F)$ is a subspace of $ L_{UKB}(E,F)$ whenever $E$ and $F$ are Banach lattices. On the other hand,  if  $F^{\prime}$ has order continuous norm, then by using Theorem 6.4 \cite{9d},
$ L_{UKB}(E,F)^+\subseteq W_{KB}(E,F)$ whenever $E$ and $F$ are Banach lattices.\\

\bibliographystyle{amsplain}

\end{document}